\documentclass[12pt, twoside]{article}
\usepackage{amsmath,amsthm,amssymb}
\usepackage{times}
\usepackage{enumerate}

\def\titlerunning#1{\gdef\titrun{#1}}
\makeatletter
\def\author#1{\gdef\autrun{\def\and{\unskip, }#1}\gdef\@author{#1}}
\def\address#1{{\def\and{\\\hspace*{18pt}}\renewcommand{\thefootnote}{}%
\footnote {#1}}%
\markboth{\autrun}{\titrun}}
\makeatother
\def\email#1{e-mail: #1}
\def\subjclass#1{{\renewcommand{\thefootnote}{}%
\footnote{\emph{Mathematics Subject Classification (2010):} #1}}}
\def\keywords#1{\par\medskip
\noindent\textbf{Keywords.} #1}

\newtheorem{thm}{Theorem}[section]
\newtheorem{lem}[thm]{Lemma}
\newtheorem{prop}[thm]{Proposition}

\newtheorem{defn}[thm]{Definition}
\newtheorem{remark}{Remark}

\theoremstyle{definition}

\numberwithin{equation}{section}

\newcommand{\Z}{{\mathbb Z}} 
\newcommand{\Q}{{\mathbb Q}}
\newcommand{\R}{{\mathbb R}}
\newcommand{\C}{{\mathbb C}}

\newcommand{\T}{{\mathbb T}}

\def\scrL{{\mathcal L}}

\def\scrN{{\mathcal N}}
\def\scrQ{{\mathcal Q}}
\def\scrS{{\mathcal S}}

\renewcommand{\i}{{\mathrm{i}}}

\newcommand{\Vol}{\operatorname{vol}}
\newcommand{\area}{\operatorname{area}}

\newcommand{\interior}{\operatorname{int}}
\newcommand{\sgn}{\operatorname{sgn}}

\newcommand{\be}{\begin{equation}}
\newcommand{\ee}{\end{equation}}
\newcommand{\bs}{\begin{split}}
\newcommand{\es}{\end{split}}
\newcommand{\bra}{\left\langle}
\newcommand{\ket}{\right\rangle}

\numberwithin{equation}{section}

\begin{document}


\baselineskip=17pt


\titlerunning{Superscars in the \v{S}eba billiard}

\title{Superscars in the \v{S}eba billiard}

\author{P\"ar Kurlberg
\and 
Henrik Uebersch\"ar}

\date{June 25, 2015}

\maketitle

\address{Department of Mathematics, KTH Royal Institute of Technology, SE-10044 \\ Stockholm, Sweden; \email{kurlberg@math.kth.se}
\and
Laboratoire Paul Painlev\'e, Universit\'e Lille 1, CNRS U.M.R. 8524
59655 Villeneuve d'Ascq Cedex​, France; \email{henrik.ueberschar@math.univ-lille1.fr}}

\thanks{P.K. was partially supported by grants from 
the G\"oran Gustafsson Foundation
for Research in Natural Sciences and
Medicine, and the Swedish Research Council (621-2011-5498).}



\begin{abstract}
  We consider the Laplacian with a delta potential (also known as a
  ``point scatterer'', or ``Fermi pseudopotential'') on an {\em
    irrational} torus, where the square of the side ratio is {\em
    diophantine}. The eigenfunctions fall into two classes --- ``old''
  eigenfunctions (75\%) of the Laplacian which vanish at the support
  of the delta potential, and therefore are not affected, and ``new''
  eigenfunctions (25\%) which are affected, and as a result feature a
  logarithmic singularity at the location of the delta potential.

  Within a {\em full} density subsequence of the new eigenfunctions we
  determine all semiclassical measures in the weak coupling regime and
  show that they are localized along $4$ wave vectors in momentum
  space --- we therefore prove the existence of so-called ``superscars''
  as predicted by Bogomolny and Schmit \cite{BogomolnySchmit}. 
  
  This result contrasts the phase space {\em equidistribution}
  which is observed for a {\em full} density subset of the new
  eigenfunctions of a point scatterer on a {\em rational} torus
  \cite{KurlbergU}.  Further, 
  in the strong coupling limit we show that a weaker form of
  localization holds for an   {\em essentially full} density
    subsequence of the new   eigenvalues; in particular quantum
  ergodicity does not hold. 

  We also explain how our results can be modified for rectangles with
  Dirichlet boundary conditions with a point scatterer in the
  interior. In this case our results extend previous work of 
  Keating, Marklof and Winn who proved
  the existence of localized semiclassical measures under a
  clustering condition on the spectrum of the Laplacian.

\end{abstract}

\section{Introduction}


In the Quantum Chaos literature the \v{S}eba billiard \cite{Seba}, a delta
potential placed inside an irrational rectangular billiard, has
attracted considerable attention
\cite{Shigehara1,Shigehara2,ShigeharaCheon3D,Shigehara3,SebaExner,Stoeckmann,BogomolnyGerlandSchmit}. \v{S}eba
introduced the model to investigate the transition between
integrability and chaos in quantum systems and numerical experiments
revealed features characteristic of  chaotic systems: level
repulsion and a Gaussian value distribution of the wave functions ---
in agreement with Berry's random wave conjecture \cite{Berry}.

The present paper deals with irrational tori having
diophantine\footnote{An irrational $\gamma$ is diophantine if there
  exist constants $C>0$, $k\geq2$ such that $|\gamma-p/q|>Cq^{-k}$ for
  any rational $p/q$.} aspect ratio; for convenience the main focus is on
periodic rather than Dirichlet boundary conditions, but the methods
also apply in the latter case (cf. Appendix A.)

The eigenfunctions of this system fall into two classes --- old and
new eigenfunctions. The old eigenfunctions are simply eigenfunctions
of the Laplacian which vanish at the position of the scatterer $x_0$
and therefore do not feel its presence. In the case of an irrational
torus they make up 75\% of the spectrum. In this paper we will only be
interested in the new eigenfunctions, which do feel the effect of the
scatterer and feature a logarithmic singularity at $x_0$. They make up
the remaining 25\% of the spectrum.

\subsection{Statement of the main result}

We prove that a full density subsequence of the new eigenfunctions of
the point scatterer fail to equidistribute in phase space in the {\em
  weak} coupling limit. Specifically, these eigenfunctions become
localized (``scarred'', or even ``superscarred'') in $4$ wave vectors
in momentum space and we are able to classify all possible
semiclassical measures which may arise along this sequence in the weak
coupling regime (i.e., fixed self-adjoint extensions).
Moreover, in the {\em strong} coupling regime (where the self-adjoint
extension parameter varies with the eigenvalue) we are able to show a
somewhat weaker result, namely that a subsequence of almost full
density fails to equidistribute in phase space.

To describe this more precisely, we first introduce semiclassical
measures arising from eigenfunctions.
\begin{defn}
  Let $\lambda$ be a new eigenvalue of the scatterer and denote by
  $g_\lambda$ the corresponding $L^2$-normalized eigenfunction. Let
  $a\in C^\infty(S^*\T^2)$ be a classical observable and let $Op(a)$
  be a zeroth order pseudo-differential operator\footnote{We will give
    a precise definition of our choice of quantization in section
    \ref{Quantisation}} associated with $a$. We define the
  distribution $d\mu_\lambda$ by the identity \be \bra
  Op(a)g_\lambda,g_\lambda\ket=\int_{S^*\T^2} a \; d\mu_\lambda.  \ee
  By the semiclassical measures for a certain sequence $\{\lambda_n\}$
  we mean the limit points of $\{d\mu_{\lambda_n}\}$ in the weak-*
  topology.
\end{defn}

By a subsequence of full density we mean the following.
\begin{defn}
Let $\scrS\subset\R$ be a countably infinite sequence of increasing numbers which accumulate at infinity. We say that $\scrS'\subset\scrS$ is a subsequence of full density if
\be
\lim_{X\to\infty}\frac{\#\{x\in\scrS' \mid x\leq X\}}{\#\{x\in\scrS \mid x\leq X\}}=1.
\ee
\end{defn}

The following definitions will be used throughout the paper.
\begin{defn}
\label{def:lattice-def}
For $a>0$ a fixed real number, define a lattice
$$\scrL_0:=\Z(a,0)\oplus\Z(0,1/a)
\subset \R^2,
$$ 
let $\scrL$ denote the dual lattice of $\scrL_0$, and let $\scrN$
denote the set of distinct Laplacian eigenvalues (i.e., squares of
norms of the lattice vectors in $\scrL$.)  Further, given $m \in
\scrN$, denote by $\lambda_m<m$ the new eigenvalue of the
scatterer
associated
with $m$. 
\end{defn}
We remark that $\{ \lambda_{m} : m \in \scrN \}$ is the {\em full} set of
new eigenvalues; this is due to a certain interlacing property, see
  Section~\ref{sec:spectrum-of-a-point-scatterer} for more details.
With notations as above, the main result of this paper is the
following result, valid for the {\em weak} coupling limit.
\begin{thm}\label{MomentumScars}
Assume that $a^4\notin\Q$ is diophantine, and consider the point
scatterer perturbation of the Laplacian on
the flat torus 
$\T^2=\R^2/2\pi\scrL_0$. There exists a subsequence $\scrN'\subset\scrN$ of full density such that
the set of semiclassical measures of the sequence $d\mu_{\lambda_m}$,
$m\in\scrN'$, is given by
the following subset of the set of probability measures on the unit cotangent 
bundle $S^*\T^2$:
\be
\scrQ = \left\{ \frac{dx}{\Vol(\T^2)}\times \frac{1}{4}(\delta_{\theta}+\delta_{-\theta}+\delta_{\pi-\theta}+\delta_{\pi+\theta})(\phi)\frac{d\phi}{2\pi} \; \Big| \; \theta\in[0,\pi/2]\right\}.
\ee
\end{thm}

\begin{remark}  As already mentioned, the result can be extended to irrational
  (diophantine) rectangles with Dirichlet boundary conditions and a
  delta potential in the interior --- the original setting of \v{S}eba's
  paper. In the appendix to this paper we illustrate how our proof
  can be modified. For a generic
  position of the scatterer we prove scarring for a proportion
  $1-\epsilon$ of all eigenfunctions, for any $\epsilon>0$. In the
  non-generic case of positions with rational coordinates, a positive
  proportion of the eigenfunctions do not feel the effect of the
  scatterer, hence are old Laplacian eigenfunctions. However, our
  theorem still applies to the new eigenfunctions associated with the
  remaining part of the spectrum.
\end{remark}

In the {\em strong} coupling
limit, which is studied in the physics literature, and in which
features such as level repulsion between the new eigenvalues are
observed, we are able to prove the following somewhat weaker result
(see Section 4 for its proof.)
\begin{thm}
Given $\delta>0$ there exists a  subsequence of the new spectrum, of
density at least $1-\delta$, 
on which the momentum representation of the new eigenfunctions carries
positive mass  on a finite 
number of points\footnote{We allow these points to depend on the
  eigenvalue.}. For $\delta$ fixed, the mass is uniformly bounded from
  below, and the number of  points is uniformly bounded from above.
\end{thm}
These results may also be easily modified for rectangular domains with
Dirichlet or Neumann boundary conditions (cf.
Remark~\ref{rem:strong-coupling-extension} in the appendix.)

\subsection{Discussion}

The scarring phenomenon described above contrasts the equidistribution
of a full density subset of new eigenfunctions for a point scatterer
on a square torus, both in the weak as well as strong coupling limits
(cf. \cite{KurlbergU}.)  Interestingly, a key feature for obtaining
equidistribution for the square torus is that the unperturbed spectrum
has unbounded multiplicities (along a generic sequence), whereas in the
diophantine aspect ratio case, where the unperturbed spectrum has bounded
multiplicities, most eigenfunctions scar  strongly. 

Moreover, the type of scarring
proven here seems quite different from the sequence of scars
established by Hassell \cite{Hassell} for the stadium billiard (his
construction is based on quasimodes corresponding to a certain sparse
sequence of ``bouncing ball modes''),
or the construction of scars for cat maps with small quantum periods
by de Bi\`evre, Faure and Nonnenmacher \cite{BFN} (they construct
sparse sequences at most half of whose mass is scarred, and a crucial
feature in the construction is having essentially {\em maximal}
spectral multiplicities; note that Bourgain has shown
\cite{Bourgain-cat} that scarring does not occur for cat maps if
multiplicities are just slightly smaller than maximal.)  We also
mention Kelmer's construction \cite{kelmer} of scars for certain
higher dimensional analogues of cat maps; here the existence of
invariant rational isotropic subspaces plays a key role.


In the original setting of the \v{S}eba billiard, i.e., for irrational
rectangles with Dirichlet boundary conditions and a delta potential in
the interior (and in the weak coupling limit),
Keating, Marklof and Winn showed \cite{KeatingMarklofWinn2} that
eigenfunctions can scar in momentum space, provided that the
unperturbed eigenfunctions are bounded from below at the location of
the scatterer, together with a certain clustering assumption on the
spectrum of the Laplacian.  (The clustering assumption is implied by
the Berry-Tabor conjecture, which suggests that the eigenvalues of a
generic integrable system behave like points from a Poisson process.)
Our proof can easily be modified for this setting.

Our results also show that contrary to the title of \v{S}eba's original
paper \cite{Seba} there is no ``wave chaos'' with respect to the wave
functions of diophantine rectangular quantum billiards (even though
chaotic effects, such as level repulsion, appear in the strong
coupling regime \cite{BogomolnyGerlandSchmit}) --- quantum
ergodicity fails, both in the weak and strong coupling regimes. Moreover, in the specific setting of \v{S}eba's original paper (weak
coupling and Dirichlet boundary conditions), we show that for any 
$\epsilon>0$ a proportion $1-\epsilon$ of the
eigenfunctions are scarred in momentum space and we determine all
possible scarred measures explicitly.

\subsection*{Acknowledgements} We would like to thank Jens Marklof, St\'ephane Nonnenmacher and Zeev Rudnick for very helpful discussions about this work. We would also like to thank the anonymous referee for his careful reading of the paper and many suggestions which led to the improvement of this paper.

\section{Background}
This section has the purpose of providing the reader with a brief
summary of various results which will be used in the paper.

\subsection{The spectrum of a point scatterer on an irrational torus}
\label{sec:spectrum-of-a-point-scatterer}
In order to realize the formal operator $$-\Delta+\alpha\delta_{x_0},
\quad (\alpha,x_0)\in\R\times\T^2$$ we use self-adjoint extension
theory. We simply state the most important facts in this section to
make the paper as self-contained as possible. For a more detailed
discussion of the theory we refer the reader to the introduction and
appendix of the paper \cite{RU}.

Recall that $\T^2=\R^2/2\pi\scrL_0$. We
restrict the positive Laplacian $-\Delta$ to the
domain $$D_0=C^\infty_c(\T^2\setminus\{x_0\})$$ of functions which
vanish near the position of the scatterer: $$H_0=-\Delta|_{D_0}$$ The
operator $H_0$ is symmetric, but it fails to be essentially
self-adjoint, in fact $H_0$ has deficiency indices $(1,1)$. Therefore
there exists a one-parameter family of self-adjoint extensions
$H_\varphi$, $\varphi\in(-\pi,\pi]$, which are restrictions of the
adjoint $H_0^*$ of the restricted operator to the domain of functions
$f\in Dom(H_0^*)$ which satisfy the logarithmic boundary
condition $$f(x)=C(\cos(\varphi/2)\frac{\log|x-x_0|}{2\pi}+\sin(\varphi/2))+o(1)$$
as $x\to x_0$ for some constant $C\in\C$. The case $\varphi=\pi$
corresponds to $\alpha=0$, i. e. we simply obtain the unrestricted
Laplacian in this case. In this paper we will study the operators
$H_\varphi$, $\varphi\in(-\pi,\pi)$. In the physics literature
\cite{Shigehara1} the operator $H_\varphi$ for fixed $\varphi$ 
is known as the ``weak coupling'' quantization of the scatterer.

Let us now focus on the special case of an irrational torus
$\T^2$. This means we take a lattice $\scrL_0$ such that
$a^4\notin\Q$. The spectrum of the operator 
$H_\varphi$ consists of two parts:\\ 
\begin{itemize}
\item[(A)] Eigenfunctions which vanish at $x_0$ and therefore do not
  ``feel'' the scatterer. These are simply eigenfunctions of the
  Laplacian, and  occur with multiplicity $m-1$ where $m$ is the
  multiplicity of the corresponding eigenspace of the Laplacian. The
  multiplicity of the positive old eigenvalues is $3$, unless the
  corresponding lattice vector lies on one of the axes, in which case
  it is $1$.\\ 

\item[(B)] Eigenfunctions which feature a logarithmic singularity at
  $x_0$.  These ``feel'' the effect of the scatterer, and turn out to be
  given by the Green's functions
  $G_\lambda=(\Delta+\lambda)^{-1}\delta_{x_0}$.  The new eigenvalues
  $\lambda$ occur with multiplicity $1$ and interlace with the ``old''
  Laplace eigenvalues (counted without multiplicity.)\\
\end{itemize} 

We will be interested in the eigenfunctions of type (B), and in
particular we will study how these eigenfunctions are distributed in
phase space as the eigenvalue tends to infinity. Recall that $\scrN$
denotes the set of distinct eigenvalues of the Laplacian on $\T^2$
(these are just norms squared of the lattice vectors in $\scrL$). For
given $n\in\scrN$ denote its multiplicity by $r(n)$.

The eigenvalues of type (B) are solutions to the equation
\be\label{weak coupling quantisation}
\sum_{n\in\scrN}r(n)\left(\frac{1}{n-\lambda}-\frac{n}{n^2+1}\right)=\tan(\varphi/2)\sum_{n\in\scrN}\frac{r(n)}{n^2+1}
\ee and they {\em interlace} with the distinct Laplacian
eigenvalues $$\scrN=\{0=n_0<n_1<n_2<\cdots\}$$ as follows \be
\lambda_{n_0}<0=n_0<\lambda_{n_1}<n_1<\lambda_{n_2}<n_2<\cdots \ee
where the new eigenvalue associated with $n\in\scrN$ is denoted by
$\lambda_n$.

\subsection{Quantization of phase space observables}\label{Quantisation}
Recall that  $\T^2=\R^2/2\pi\scrL_0$ where $\scrL_0=\Z(a,0)\oplus \Z(0,1/a)$,
$a>0$, and $\scrL$ denotes its dual lattice. Consider a classical
symbol $a\in C^\infty(S^*\T^2)$, where $S^*\T^2\simeq \T^2\times S^1$
denotes the unit cotangent bundle of $\T^2$. We may expand $a$ in the
Fourier series
\begin{equation}
a(x,\phi)=\sum_{\zeta\in\scrL,k\in\Z}\hat{a}(\zeta,k)e^{\i \left\langle \zeta,x \right\rangle+\i k\phi}.
\end{equation}
We choose the following quantization of the symbol $a$. Let $f\in L^2(\T^2)$ with Fourier expansion
\begin{equation}
f(x)=\sum_{\xi\in\scrL}\hat{f}(\xi)e^{\i\left\langle \xi,x \right\rangle}.
\end{equation}
On the Fourier side the action of the $0$-th order pseudodifferential
operator $Op(a)$ is defined by (we have chosen a ``right''
quantization, which means we first apply momentum then position
operators (cf. section 2.1, in \cite{KurlbergU}))
\begin{equation}
\widehat{(Op(a)f)}(\xi)=\sum_{\zeta\neq\xi\in\scrL,k\in\Z}\hat{a}(\zeta,k)\left(\frac{\tilde{\xi}-\tilde{\zeta}}{|\xi-\zeta|}\right)^k \hat{f}(\xi-\zeta)+\sum_{k\in\Z}\hat{a}(\xi,k)\hat{f}(0), 
\end{equation}
where for a given $\xi=(\xi_1,\xi_2)\in\scrL$ we define
$\tilde{\xi}:=\xi_1+\i\xi_2$. 

In terms of the Fourier coefficients the matrix elements of $Op(a)$ can be written as
\begin{equation}
\left\langle Op(a)f,f \right\rangle=\sum_{\xi\in\scrL} \widehat{(Op(a)f)}(\xi)\overline{\hat{f}(\xi)}.
\end{equation}
With $e_{\zeta,k}(x,\phi):=e^{\i\left\langle \zeta,x \right\rangle+\i
  k\phi}$, we then have
\begin{equation}
\left\langle Op(e_{\zeta,k})f,f \right\rangle=\sum_{\xi\in\scrL\setminus\{\zeta\}} \left(\frac{\tilde{\xi}-\tilde{\zeta}}{|\xi-\zeta|}\right)^k\overline{\hat{f}(\xi)}\hat{f}(\xi-\zeta)+\overline{\hat{f}(\zeta)}\hat{f}(0).
\end{equation}

\subsubsection{Mixed modes}
If $\zeta\neq0$ we have the bound
\begin{equation}
|\left\langle Op(e_{\zeta,k})f,f \right\rangle|\leq\sum_{\xi\in\scrL} |\hat{f}(\xi)||\hat{f}(\xi-\zeta)|.
\end{equation}
In the case $f=g_\lambda=G_\lambda/\|G_\lambda\|_2$ we have the $L^2$-expansion $$G_\lambda(x,x_0)=-\frac{1}{4\pi^2}\sum_{\xi\in\scrL} c(\xi)e^{\i\left\langle x-x_0,\xi \right\rangle}$$
where $c(\xi)=\frac{1}{|\xi|^2-\lambda}$.
We obtain
\begin{equation}
|\left\langle Op(e_{\zeta,k})g_\lambda,g_\lambda \right\rangle|\leq
\frac{\sum_{\xi\in\scrL} |c(\xi)||c(\xi-\zeta)|}{\sum_{\xi\in\scrL} |c(\xi)|^2}.
\end{equation}
In \cite{RU} it was proved that one can construct a full density
subsequence $\scrN'\subset\scrN$ such that for any nonzero lattice
vector $\zeta\in\scrL$ the matrix elements of $Op(e_{\zeta,k})$ vanish
as $n\to\infty$ along $\scrN'$. The following result was
obtained.
\begin{thm}\label{mixed}{\bf (Rudnick-U., 2012)}
Let $\scrL$ be a unimodular lattice as above. There exists a subsequence $\scrN'\subset\scrN$ of full density such that for any $\zeta\in\scrL$, $\zeta\neq0$, $k\in\Z$
\begin{equation}
\lim_{\substack{n\to\infty \\ n\in\scrN'}}\left\langle Op(e_{\zeta,k})g_{\lambda_n},g_{\lambda_n}\right\rangle=0.
\end{equation}
\end{thm}
\begin{remark}
Although the paper \cite{RU} is solely concerned with the weak coupling regime, i.e. fixed self-adjoint extensions, the theorem holds generally for Green's functions $G_{\lambda_n}$ and any sequence of real numbers $\{\lambda_n\}$ which interlaces with the Laplacian eigenvalues on $\T^2$. A detailed explanation is given in \cite{KurlbergU}, p. 7, Remark 3.
\end{remark}

\subsubsection{Pure momentum modes}
Let us consider the case $\zeta=0$. We rewrite the matrix elements as
\begin{equation}\label{pure momentum matrix element}
\begin{split}
\left\langle Op(e_{0,k})g_\lambda,g_\lambda\right\rangle=&\frac{\sum_{\xi\in\scrL\setminus\{0\}} (\tilde{\xi}/|\xi|)^k|c(\xi)|^2+|c(0)|^2}{\sum_{\xi\in\scrL}|c(\xi)|^2}\\
=&\frac{\sum_{n\in \scrN}\frac{w_{k}(n)}{(n-\lambda)^2}}{\sum_{n\in \scrN}\frac{r(n)}{(n-\lambda)^2}}
\end{split}
\end{equation}
where for $0\neq n\in\scrN$ we define the exponential sum
\begin{equation}
w_{k}(n):=\sum_{\substack{|\xi|^2=n \\ \xi\in\scrL\setminus\{0\}}}\left(\frac{\tilde{\xi}}{|\xi|}\right)^k,
\end{equation}
and for notational convenience  we set $w_k(0):=1$.

\subsection{Pair correlations for values of quadratic
  forms}
In this section we will briefly review a result of Eskin, Margulis and
Mozes \cite{EMM} on the pair correlations of the values of the
quadratic form $Q(k,l)=a^{-2}k^2+a^2 l^2$, $k,l\in \Z$, where
$\gamma=a^4$ is diophantine.  We have the following theorem
(cf. Thm. 1.7 in \cite{EMM}), which we only state in the special case
relevant to the present paper. Note that $\area(\T^2)/4\pi=\pi$.
\begin{thm}\label{EMM} {\bf(Eskin-Margulis-Mozes, 2005)}
Let $\gamma=a^4$ be diophantine and $0\notin(b,c)$. Denote the Laplacian eigenvalues on $\T^2=\R^2/2\pi\scrL_0$ by $\{\lambda_j(\T^2)\}$. 
Then
\be
\lim_{X\to\infty}\frac{\#\{\lambda_j(\T^2),\lambda_k(\T^2)\leq X \mid \lambda_j(\T^2)-\lambda_k(\T^2)\in(b,c)\}}{X}=\pi^2(c-b).
\ee
\end{thm}
The theorem above proves the Berry-Tabor conjecture \cite{BerryTabor} for the pair correlations of the Laplacian eigenvalues on the torus $\T^2$, where $\gamma=a^4$ is diophantine. Recall that the Laplacian eigenvalues are given by the squared norms $(k^2+a^4l^2)/a^2$ and the ordered set of such distinct squared norms is denoted by $\scrN$. 

As in the irrational case the multiplicities of the Laplacian
eigenvalues on the torus are generically $4$, we have for the pair
correlation of the distinct Laplacian eigenvalues, i.e. the set of
norms $\scrN$, 
\begin{equation}
  \label{eq:norm-pair-correlation}
\lim_{X\to\infty}\frac{\{m,n\in\scrN \mid m,n\leq X, \; m-n\in(b,c)\}}{X}
=\frac{\pi^2}{16}(c-b).
\end{equation}
Letting 
$$
\scrN(X) := \{n\in\scrN \mid n\leq X\}
$$
denote the intersection of $\scrN$ and the interval $[0,X]$, we have
the counting asymptotic (``Weyl's law'')
\begin{equation}
  \label{eq:Weyl-law}
|\scrN(X) |
\sim \frac{\pi}{4}X
\end{equation}
as $X\to\infty$. Consequently, we
obtain $$\lim_{X\to\infty}\frac{1}{|\scrN(X)|}\#\{m,n\in\scrN(X) \mid
m-n\in(b,c)\}=\frac{\pi}{4}(c-b).$$ 
We note that the mean spacing is
$4/\pi$ (cf. eq. \eqref{eq:Weyl-law}).


\section{The weak coupling limit --- proof of Theorem
  \ref{MomentumScars}}
We begin by proving the following proposition.
\begin{prop}\label{QEfails}
  Let $\scrL$ be a diophantine rectangular unimodular lattice as
  above\footnote{In particular, $a^{4} \not \in \Q$ is diophantine; cf.
    Definition~\ref{def:lattice-def}.}. There exists 
  a subsequence $\scrN'\subset\scrN$ of full density such that for
  $m\in\scrN'$ and any integer $k$,
\be\label{irrationalpuremode} \bra
  Op(e_{0,k})g_{\lambda_m},g_{\lambda_m} \ket =
  \frac{w_{k}(m)}{r(m)}+o(1) \ee as $m\to\infty$ along
  $\scrN'$.
\end{prop}

Before  giving the proof of Proposition \ref{QEfails} we recall
the following bound from \cite{RU2}; it shows that, in the {\em weak}
coupling regime, the new 
eigenvalues of the scatterer and the eigenvalues of the Laplacian
generically ``clump'' together.
\begin{thm}\label{clumping}
Let $\scrL$ be an irrational lattice as above. 
Given any increasing function $f$ such that $f(m)\to\infty$ as
$m\to\infty$ along $\scrN$ there exists a density
one subsequence $\scrN''\subset\scrN$ such that for all $m \in
\scrN''$, 
\be
0 < m-\lambda_m \ll \frac{f(m)}{\log m}
\ee
\end{thm}

The following key  Lemma will allow us  to ``circumvent'' the lack of
uniformity in the size of the interval $(b,c)$ in Theorem~\ref{EMM}.
\begin{lem}
\label{lem:difference-bigger-than-three}
For $A \geq 3$ we have
$$
\sum_{\substack{ m,n \in \scrN(x) \\ |m-n|> A}}
\frac{1}{|m-n|^{2}}
\ll \frac{x}{A}
$$  
\end{lem}
\begin{proof}
Given an integer $k \geq 0$, define
$$
M(k) := |\{ n \in \scrN : n \in [k,k+1]  \}|
$$
We begin by deducing an $L^{2}$ bound on $M(k)$ using Theorem~\ref{EMM}:
\begin{multline*}
\sum_{k \leq T} M(k)^{2} =
\sum_{k \leq T}
|\{ m,n \in \scrN : m,n \in [k,k+1]  \}|
\\ \leq
|\{ m,n \in \scrN : m,n \leq T+1, m-n \in [-1,1]  \}|
\end{multline*}
which,  by Theorem~\ref{EMM}, is
$$
\ll
2(T+1) + (T+1) \ll T
$$
(note that we include pairs $m,n \leq T+1$ for which $m=n$; this gives
rise to the additional $T+1$ term.)

Using Cauchy-Schwarz, we now find that for $l \ll T$, we have
\begin{equation}
  \label{eq:shift-bound}
\sum_{k \leq T} M(k)M(k+l) \leq
\left(\sum_{k \leq T} M(k)^{2} \right)^{1/2} \cdot
\left( \sum_{k \leq T+l} M(k)^{2}  \right)^{1/2}
\ll T^{1/2} \cdot T^{1/2} = T
\end{equation}

We may now conclude the proof:
$$
\sum_{\substack{ m,n \in \scrN(x) \\ |m-n|> A}}
\frac{1}{|m-n|^{2}}
\ll
\sum_{k = A}^{x}
\frac{|\{m,n \in \scrN(x) : m < n, n-m \in [k,k+1] \}}{k^{2}}
$$
$$
\leq
\sum_{k = A}^{x}
\frac{1}{k^{2}}  
\sum_{l \leq x } M(l) \cdot (M(l+k)+M(l+k+1)  )
$$
which, by using (\ref{eq:shift-bound}), is 
$$
\ll
\sum_{k = A}^{x}
\frac{x}{k^{2}}  \ll \frac{x}{A}
$$

\end{proof}

We can now prove the following key estimate.
\begin{prop}
\label{prop:key-prop}
There exists a subsequence $\scrN_{1}\subset\scrN$ of full density such
that for $m\in\scrN_{1}$ 
\begin{equation}
  \label{eq:diff-squared-sum}
\sum_{ n\in\scrN, \; n\neq m }\frac{1}{(n-m)^2} = O((\log m)^{2-\epsilon}).
\end{equation}
\end{prop}
\begin{proof}
  Let $G(m)=(\log m)^{-1+\epsilon}$ for a small fixed
  $\epsilon\in(0,1)$, and denote by $m_-$, $m_+$ the nearest neighbours to the left and right of $m\in\scrN$.
  We claim that the subsequence 
$$
\scrN_0=\{m\in  \scrN \mid |m-m_-|,|m-m_+| \geq G(m)\}
$$ 
is of full density in  $\scrN$. Let us assume for a contradiction that
the sequence 
$$
\scrN_0' =\{ m\in\scrN \mid |m-m_-|\,\text{or}\,|m-m_+|<G(m)\}
$$ 
is
of non-zero density, i.e., that for some $\eta>0$,
\be
\label{lower_pos} 
|\scrN_0'(x)|\geq  \eta  |\scrN(x)|
\ee 
holds for a sequence of values of $x$ tending to infinity. Recall
that $N(x)\sim 
\frac{\pi}{4}x$, as 
$x\to\infty$, since $\area(\T^2)=4\pi^2$ and the multiplicity is
generically $4$.
Using Theorem~\ref{EMM}, we thus find that as $x\to\infty$,
\be
\begin{split}
&\frac{1}{|\scrN(x)|}\#\{ m\in \scrN_0' \mid m\leq x \}\\
\leq \; &\frac{1}{|\scrN(x)|}
\#\{ m,n \in\scrN \mid |m-n|< G(m), \quad m\neq n, \quad m,n\leq x\}\\
\leq \; &\frac{1}{|\scrN(x)|}
\#\{ m,n \in\scrN \mid |m-n|\leq\frac{\eta}{\pi}, \quad m\neq n, \quad x^{1/4}\leq m,n\leq x\}\\ 
&\hspace{70mm}+O(x^{-1/2})\\
\to \; &\frac{\eta}{2}
\end{split}
\ee 
which leads to a contradiction to eq. \eqref{lower_pos}.  

Next we estimate the sum on the LHS of \eqref{eq:diff-squared-sum}. 
We first note that for
$m\in\scrN(x)$ and $x$ large we have 
\be
\begin{split}
\sum_{\substack{n\in\scrN, n\neq m}}\frac{1}{(m-n)^2}=
&\sum_{\substack{n\in\scrN(2x), n\neq    m}}
\frac{1}{(m-n)^2}+\sum_{\substack{  n\in\scrN,n>2x, n\neq m }}\frac{1}{(m-n)^2}
\\ 
=&\sum_{\substack{ n\in\scrN(2x), n\neq m }}
\frac{1}{(m-n)^2}+O\left(\frac{1}{x}\right)
\end{split}
\ee
where we used in the last line that $m\leq x$ and $n>2x$, hence
$n-m>n/2$ and the bound on the second sum follows from Weyl's law
(see \eqref{eq:Weyl-law}) and partial summation.

Next we show that there exists a density one subsequence 
$ \scrN_{1} \subset \scrN_{0}$
such that for all $m \in \scrN_{1}(x)$,
$$
\sum_{\substack{n\in\scrN(2x), n\neq m}}\frac{1}{(m-n)^2} \ll (\log m)^{2-\epsilon}.
$$
We have 
\begin{multline}
\sum_{m \in \scrN_0(x)}
\sum_{\substack{n\in\scrN(2x)}}
\frac{1}{(m-n)^2} 
\leq
\sum_{m\in\scrN(x)}\sum_{\substack{n\in\scrN(2x) \\ |n-m|\geq  G(m)}}
\frac{1}{(m-n)^2} 
\\
=\sum_{m\in\scrN(x)}\sum_{\substack{n\in\scrN(2x) \\ |n-m|\in[G(m),1]}}\frac{1}{(m-n)^2} 
+\sum_{m\in\scrN(x)}\sum_{\substack{n\in\scrN(2x) \\ |n-m|>1}}\frac{1}{(m-n)^2}
\end{multline}
We estimate the second sum by
\be
\begin{split}
&\sum_{m\in\scrN(x)}\sum_{\substack{n\in\scrN(2x) \\ |n-m|>1}}\frac{1}{(m-n)^2}\\
&\leq \sum_{k\leq 2x} \frac{1}{k^2}\#\{m\in\scrN(x),n\in\scrN(2x) \mid |m-n|\in[k,k+1)\}\\
&=:\sum_{k\leq 2x} \frac{c(k)}{k^2}\ll \log x
\end{split}
\ee
where the logarithmic bound follows from
$$
\sum_{k\leq 2x}c(k)
\leq |\{ m,n \in \scrN(2x) : |m-n| \leq 2x+2 \} |    \leq
|\scrN(2x)|^{2} 
\ll x^2
$$ 
together with summation by parts.

For the first sum we have, by Theorem~\ref{EMM},
\be
\begin{split}
 & \sum_{m\in\scrN(x)}\sum_{\substack{n\in\scrN(2x) \\
     |n-m|\in[G(m),1]}}\frac{1}{(m-n)^2}\\ 
 \ll \; & \frac{1}{G(x)^2}\#\{m\in\scrN(x),n\in\scrN(2x) \mid
 |m-n|\in (0,1]\} \\ 
 \ll \; & \frac{x}{G(x)^2} \ll x(\log x)^{2-2\epsilon}.
\end{split}
\ee

Now, let 
$$
F(m)=\sum_{\substack{n\in\scrN(2x) \\ |n-m|\geq
    G(m)}}\frac{1}{(m-n)^2}.
$$ 
From the estimates above we have, for fixed $\delta\in(0,1)$,
\be
\sum_{m\in\scrN_0(x),m\geq x^\delta}F(m)\ll x(\log x)^{2-2\epsilon}
\ee

Letting $T(m)=(\log m)^{2-\epsilon}$ and using Chebyshev's inequality we
find that 
\begin{multline}
\#\{m\in\scrN_0(x) \mid F(m)\geq T(m), m\geq x^\delta\} \\
\ll \; T(x)^{-1}\sum_{m\in\scrN_{0}(x), m\geq x^\delta} F(m) 
\ll \; x (\log x)^{2-2\epsilon}/T(x)=x/(\log x)^\epsilon
\end{multline}
where we have used that $T(m)\asymp T(x)$ for $m\in[x^\delta,x]$. It
follows that 
$F(m)<T(m)$ is a density one condition inside $\scrN_{0}(x)$ thereby
concluding the proof.
\end{proof}

\subsection{Proof of Proposition \ref{QEfails}}
\begin{proof}
Fix some integer $k\neq0$. In order to construct the full density
subsequence we will use the result about ``clumping'' of the spectrum as
stated in Theorem \ref{clumping}. 
In what follows assume that $f$ is as in Theorem \ref{clumping}.
By Theorem~\ref{clumping}, we have for $m\in\scrN''$ that $|m
-\lambda_{m}|^{2} = O( f(m)^{2}/\log^{2}m)$, and hence
\be
\bra Op(e_{0,k})g_\lambda,g_\lambda\ket = \frac{w_{k}(m)+O((\log
  m)^{-2}f(m)^2)\sum_{\substack{n\neq m \\
      n\in\scrN}}w_{
      k}(n)(n-\lambda_m)^{-2}}{r(m)+O((\log
  m)^{-2}f(m)^2)\sum_{\substack{n\neq
      m\\n\in\scrN}}r(n)(n-\lambda_m)^{-2}} 
\ee
Let $\scrN'=\scrN_1\cap\scrN''$, with $\scrN_1$ as in
Proposition~\ref{prop:key-prop}. For $m\in\scrN'$, by 
the proof of Proposition~\ref{prop:key-prop} (in particular note that
$|m-m_-|,|m-m_+| \geq G(m) = (\log m)^{-1+\epsilon}$ holds for $m \in
\scrN_1 \subset 
\scrN_0$), if we take $f(m) = \log \log m$, then
\be
\sum_{\substack{n\neq m \\
    n\in\scrN}}\frac{|w_{k}(n)|}{(n-\lambda_m)^2} \ll
\sum_{\substack{n\neq m, \; n\in\scrN \\ |n-m| \geq
    G(m)}}\frac{1}{(n-m)^2}=O((\log m)^{2-\epsilon}) 
\ee
and
it follows 
that
\begin{multline*}
\bra Op(e_{0,k})g_\lambda,g_\lambda\ket=
\frac{w_{k}(m) +O((\log m)^{-\epsilon}) \cdot f(m)^2}
{r(m) +O((\log m)^{-\epsilon}) \cdot f(m)^2}
=\frac{w_{k}(m)+o(1)}{r(m)+o(1)}
\\
=\frac{w_{k}(m)}{r(m)}+o(1)
\end{multline*} 
as $m\to\infty$
(note that $|w_{k}(m)| \leq r(m) \leq 4$.) So the identity
\eqref{irrationalpuremode} follows.  
\end{proof}

Proposition \ref{QEfails} easily gives a classification of the quantum
limits which may arise within the sequence $\scrN'$. We are interested
in the sequence $d\mu_{\lambda_m}$, $m\in\tilde{\scrN}$, where
$\tilde{\scrN}$ denotes the intersection of the subsequences in
Theorems \ref{mixed} and \ref{clumping}, so $\tilde{\scrN}$ is of full
density. We would like to determine the quantum limits of this
sequence, i.e. the limit points in the weak-* topology.

From Theorem \ref{mixed} we know that along $\tilde{\scrN}$ the limit measures must be
flat, or equidistributed, 
  in position.  Moreover, from Proposition
\ref{QEfails} we also know that the matrix elements of pure momentum
observables for the eigenfunction $g_{\lambda_m}$,
$m\in\tilde{\scrN}$, tend to stay away from zero, because for an
irrational lattice $\scrL$ the multiplicity $r(n)$ is
bounded. The intuition is that the sequence $d\mu_{\lambda_m}$,
$m\in\tilde{\scrN}$, becomes localized in momentum in the
semiclassical limit. Theorem \ref{MomentumScars} determines the set of such
localized quantum limits.

\subsection{Proof of Theorem \ref{MomentumScars}}
Consider the classical observable $$a=\sum_{\zeta\in\scrL, k\in\Z}\hat{a}(\zeta,k)e_{\zeta,k}.$$ Let $m\in \tilde{\scrN}$. By a standard diagonalization argument (see section 4 in \cite{KurlbergU}) it suffices to prove the result for the trigonometric polynomials $$P_J=\sum_{\substack{\zeta\in\scrL, k\in\Z \\ |\zeta|,|k|\leq J}}\hat{a}(\zeta,k)e_{\zeta,k}.$$

It follows from \eqref{irrationalpuremode} and Theorem \ref{mixed} that 
\be
\bra Op(P_J)g_{\lambda_m},g_{\lambda_m} \ket = \frac{1}{r(m)}\sum_{2|k}\hat{a}(0,k)w_{k}(m)+o(1)
\ee
as $m\to\infty$. For given $m\in\scrN$ let $\theta_m\in[0,\pi/2]$ be the phase angle of the lattice point on the upper right arc of the circle $|\xi|^2=m$, i.e. $\hat{\xi}=m^{1/2}e^{\i\theta_m}$ for some $\xi\in\scrL$. Since $\scrL$ is irrational we have 
$$\frac{w_{k}(m)}{r(m)}=
\begin{cases}
\cos(k\theta_m), \quad 2|k,\\
\\
0, \quad \text{otherwise.}
\end{cases}$$ 

We have the following Lemma.

\begin{lem}
The sequence of angles $\{\theta_m\}_{m\in\tilde{\scrN}}$ is dense in $[0,\pi/2]$.
\end{lem}
\begin{proof}
Let $I \subset [0,\pi/2]$  be a nonempty open interval.  As
\begin{multline*}
|\{m\in\scrN(x) \mid \theta_m\in I\}|
=
|\{\xi\in\scrL : |\xi|^2\leq x, \, \xi_1, \xi_{2} \geq 0, \text{ and}
\arctan(\xi_2/\xi_1)\in I\}|, 
\end{multline*}
and the latter can be interpreted as the number of $\Z^2$-lattice
points inside the intersection of an ellipse with a circular sector,
dilated by $\sqrt{x}$, we find that
$|\{m\in\scrN(x) \mid \theta_m\in I\}| \sim c_{I} \cdot x$ as $x \to
\infty$, for some $c_{I} >0$.  Since the interval $I$ can be freely
chosen, the result follows. 
\end{proof}

The set of limit points of the sequence $(\bra
Op(P_J)g_{\lambda_m},g_{\lambda_m} \ket)_{m\in\tilde{\scrN}}$ is
thus given by
$$
\left\{\sum_{2|k}\cos(k\theta)\hat{a}(0,k) \Big| \theta\in[0,\pi/2]\right\}.$$
%
Now, since for $k$  even,
$$\cos(k\theta)=\frac{1}{4}(e^{\i k\theta}+e^{-\i k\theta}+e^{\i k(\pi+\theta)}+e^{\i k(\pi-\theta)}),$$
we find that all such limit set  elements
can  be rewritten as
\begin{multline}
\sum_{2|k} \cos(k\theta)\hat{a}(0,k)
=\int_{S^*\T^2}a(x,\phi) \sum_{2|k}\cos(k\theta)e^{-\i k\phi} \frac{dx\;d\phi}{\Vol(S^*\T^2)}\\
=\int_{S^*\T^2}a(x,\phi) \frac{dx}{\Vol(\T^2)}
 \quad \times\frac{1}{4}(\delta_\theta+\delta_{-\theta}+\delta_{\pi+\theta}+\delta_{\pi-\theta})(\phi) \frac{d\phi}{2\pi},
\end{multline}
and the proof is concluded.

\section{The strong coupling limit}
\label{sec:strong-coupl-limit}

Let $\lambda_{m}$ be any interlacing sequence.  Then there exists a
{\em positive density} subsequence $\scrN' \subset \scrN$ such that
$\{G_{\lambda_{m}}\}_{m \in \scrN'}$ {\em does not equidistribute}.

As before, for notational convenience we define $\scrN(T) := \{ n \in
\scrN : n \leq T \}$.
%
%
Further, let $n_{1}, \ldots, n_{k}, \ldots$ be ordered representatives of the
elements in the set $\scrN$ (i.e., so that $n_{1} < n_{2} < \ldots$),
and let $s_{i} := n_{i+1} -n_{i}$ denote the consecutive spacings.

\begin{lem}
\label{lem:many-small-gaps}
The number of $i \leq T$ such that $s_{i} > G > 0$ is $\leq  T/G \cdot (4/\pi+o(1)) $.
\end{lem}
\begin{proof}
Recalling that $\scrN(T) \sim T \cdot \pi/4 $, we find that  $\sum_{i \leq T}
s_{i} = (1+o(1))\cdot 4T/\pi$.  Since $s_{i} \geq 
0$ for all $i$, the statement is an immediate consequence of
Chebychev's inequality. 
\end{proof}

\begin{lem}  
\label{short-interval-multiplicity-bound}
Given $D>0,E \geq 1$,
$$
|\{n \in \scrN(T) :  |\scrN(T) \cap [n-D,n+D]| > E+1       \}|
\ll
\frac{DT}{E}.
$$  
\end{lem}
\begin{proof}
By Theorem~\ref{EMM} (see (\ref{eq:norm-pair-correlation})),
\begin{multline}
\sum_{n \in \scrN(T)} (|\scrN(T) \cap [n-D,n+D]| -1)
\\=
|\{n,m \in \scrN(T) : m \neq n, |m-n| \leq D \}|
\sim  \pi^2/16 \cdot 2D \cdot T
\end{multline}
and hence, by Chebychev, 
$$
|\{n \in \scrN(T) :  |\scrN(T) \cap [n-D,n+D]| > E+1       \}| \ll 
\frac{DT}{E}
$$
\end{proof}

We can now finish the proof.  Define $\scrN'$ as follows: for $G$
large take $n \in
\scrN$ such that the gap to the nearest left neighbour is at most $G$;
by  Lemma~\ref{lem:many-small-gaps} this sequence has density at least 
$1-2/G$.

By Chebychev's inequality and
Lemma~\ref{lem:difference-bigger-than-three} (note that the
Lemma is valid also in the
stroung coupling limit) we may chose $F$
sufficiently large so that 
$$
|\{ m \in \scrN(T) :  \sum_{n \in \scrN(T) : |m-n| > 3}
\frac{1}{|m-n|^{2} } > F
\}| \leq T/G; 
$$
remove  all such
$m$ and we are left with a sequence of density at least $1-3/G$.

Next take $D=3$ in Lemma~\ref{short-interval-multiplicity-bound}, and
  choose $E$ sufficiently large so that
$$
|\{n \in \scrN(T) :  |\scrN(T) \cap [n-D,n+D]| > E  +1     \}|
\leq T /G;
$$
removing also these elements we are left with a set of density at least
$1-4/G$.

Now, for $m \in \scrN'$ we have the following:
\begin{enumerate}
\item $|\lambda_{m}-m| \leq G$,
\item $|\{ n \in \scrN: 0 < |m-n| \leq 3 \}| \leq E$
\item $\sum_{n \in \scrN : |m-n| > 3} \frac{1}{(m-n)^{2}} \leq F$
\end{enumerate}
Thus, if we consider pure momentum observables, and given $n \in
\scrN$ we let $\mu_n$ denote the measure on the unit circle consisting
of four delta measures (corresponding to lattice points lying on a
circle of radius $\sqrt{n}$), we find that the measure --- not
necessarily a probability measure since we have not yet normalized ---
associated with $G_{\lambda_{m}}$ is given by
$$
\sum_{n \in \scrN}  \frac{\mu_{n}}{(n-\lambda_{m})^{2}}
=
\frac{\mu_{m}}{(m-\lambda_{m})^{2}}
+
\sum_{n \in \scrN : 0< |n-m| \leq 3}  \frac{\mu_{n}}{(n-\lambda_{m})^{2}}
+
\sum_{n \in \scrN : |n-m| > 3}  \frac{\mu_{n}}{(n-\lambda_{m})^{2}}
$$
where the first term is $\gg 1/G^{2}$, the second sum has at most $E$
terms, and the last sum is $\ll F$.
In particular, for $G$ fixed the mass contribution from the first two
terms
$$
\frac{\mu_{m}}{(m-\lambda_{m})^{2}}
+
\sum_{n \in \scrN : 0< |n-m| \leq 3}  \frac{\mu_{n}}{(n-\lambda_{m})^{2}}
$$
is uniformly bounded from below, and the number of terms in the sum is
uniformly bounded from above.
Hence the finite sum
$$
\sum_{n \in \scrN : |n-m| \leq 3}
\frac{\mu_{n}}{(n-\lambda_{m})^{2}}
$$
carries mass uniformly bounded from below; after normalizing so that
we obtain a probability measure, we find that the normalized measure
will have a positive proportion of its mass on a {\em finite} number
of points. 

\begin{appendix}
\section{Dirichlet boundary conditions}

In \v{S}eba's original paper \cite{Seba} the author considers an irrational rectangle $D$ with a delta potential placed in the interior of $D$ and Dirichlet boundary conditions. The setting of the torus has the advantage that calculations are much simplified because of translation invariance, i. e. the position of the potential is not important. The subject of this appendix is to illustrate how our proof can easily be modified for this setting. A modification would work in the analogous and correspond to a different character in the Fourier representation of the eigenfunctions.

\subsection{The spectrum and eigenfunctions}

Let $D=[0,2\pi a]\times[0,2\pi/a]$, $a^4\notin\Q$ diophantine. Let $z\in\interior{D}$. We study the self-adjoint extensions of the restricted Dirichlet Laplacian $-\Delta|_{D_0}$, where $D_0=\{f\in C^\infty_c(D\setminus\{z\}) \mid f|_{\partial D}=0\}$. This operator has deficiency indices $(1,1)$ and we denote the one parameter family of self-adjoint extensions by $\{-\Delta_\varphi^D\}_{\varphi\in(-\pi,\pi]}$. 

The eigenfunctions of $-\Delta_\varphi^D$ are given by the Green's functions
$$G_\lambda^D(x)=\sum_{\substack{\xi\in\scrL \\
    \xi_1,\xi_2>0}}\frac{\psi_\xi(x)\overline{\psi_\xi(z)}}{|\xi|^2-\lambda},
\quad \psi_\xi(x)=\frac{1}{\pi^2}\sin(\xi_1 x_1)\sin(\xi_2 x_2)$$
where $\scrL=\Z(1/a,0)\oplus\Z(0,a)$. The eigenvalues are solutions of
the equation 
\begin{equation}
\sum_{\substack{\xi\in\scrL \\ \xi_1,\xi_2>0}}|\psi_\xi(z)|^2
\left\{\frac{1}{|\xi|^2-\lambda}-\frac{1}{|\xi|^2+1}\right\}=C_{\scrL}\tan\left(\frac{\varphi}{2}\right)
\end{equation}
where $C_{\scrL}=\sum_{\xi\in\scrL}\frac{1}{|\xi|^4+1}$.

\subsection{Modification of the proof of Theorem \ref{MomentumScars}}

We can rewrite the function $G_\lambda^D$ as
\begin{equation}
\sum_{\xi\in\scrL}\frac{\chi(\xi)\overline{\psi_\xi(z)}
e^{\i\left\langle x,\xi \right\rangle}}{|\xi|^2-\lambda}, \quad \text{where\;} \chi(\xi)=\sgn(\xi_1)\sgn(\xi_2)
\end{equation}
and 
\begin{equation*}
\sgn(x)=
\begin{cases}
1, \quad \text{if\;} x>0\\
0, \quad \text{if\;} x=0\\
-1, \quad \text{if\;} x<0.
\end{cases}
\end{equation*}
\begin{proof}
To see this, first of all define, given $\xi=(\xi_1,\xi_2)$, $\xi_1,\xi_2>0$, define $\bar{\xi}=(\xi_1,-\xi_2)$.
We may expand the Laplacian eigenfunctions into complex exponentials
\begin{equation}
\begin{split}
\psi_\xi(x)=&\frac{1}{\pi^2}\sin(\xi_1 x_1)\sin(\xi_2 x_2)\\
=&-\frac{1}{4\pi^2}(e^{\i\xi_1 x_1}-e^{-\i\xi_1 x_1})(e^{\i\xi_2 x_2}-e^{-\i\xi_2 x_2})\\
=&-\frac{1}{4\pi^2}\sum_{\eta=\xi,-\xi,\bar{\xi},-\bar{\xi}}\chi(\eta)e^{\i\left\langle \eta,x \right\rangle}.
\end{split}
\end{equation}
Hence, we obtain (noting $\psi_\xi=0$ if $\xi_1\xi_2=0$)
\begin{equation}
\begin{split}
G_\lambda^D(x)
=&\sum_{\substack{\xi\in\scrL \\ \xi_1,\xi_2>0}}\frac{\psi_\xi(x)\overline{\psi_\xi(z)}}{|\xi|^2-\lambda} \\
=&-\frac{1}{4\pi^2}\sum_{\substack{\xi\in\scrL}}\frac{\chi(\xi)\overline{\psi_\xi(z)}e^{\i\left\langle \xi,x \right\rangle}}{|\xi|^2-\lambda}.
\end{split}
\end{equation}
\end{proof}

Let $g_\lambda^D=G_\lambda^D/\|G_\lambda^D\|_2$
and define $$d(\xi):=\frac{\chi(\xi)\overline{\psi_\xi(z)}}{|\xi|^2-\lambda}.$$
We then obtain for the matrix element of a pure momentum monomial 
$e_{0,k}$ that
\begin{equation}\label{dirichlet pure momentum matrix element}
\begin{split}
\left\langle Op(e_{0,k})g_\lambda^D,g_\lambda^D\right\rangle
=&
\frac{\sum_{\xi\in\scrL\setminus\{0\}} 
(\tilde{\xi}/|\xi|)^k|d(\xi)|^2+|d(0)|^2}{
\sum_{\xi\in\scrL_{0}}|d(\xi)|^2} \\
=&
\frac{ 
\sum_{n\in \scrN}\frac{\delta_n w_k(n)}{(n-\lambda)^2}}
{ 
\sum_{n\in \scrN}\frac{\delta_n \cdot r(n)}{(n-\lambda)^2}}
\end{split}
\end{equation}
where $\delta_n=|\psi_{\xi(n)}(z)|^2$ and $\xi(n)\in\scrL$ is the lattice
vector which solves the equation $|\xi|^2=n,\;\xi_1,\xi_2 \geq 0$,
and
 $$w_k(n)=\sum_{\substack{\xi\in\scrL \\
    |\xi|^2=n}}\left(\frac{\tilde{\xi}}{|\xi|}\right)^k.$$ 

{\bf Assumption:} Suppose that the position $z\in\interior{D}$ is
``generic''\footnote{
In the case of rational coordinates there will be a positive proportion of eigenvalues whose eigenfunctions vanish at the position of the scatterer and therefore do not feel its effect. In the generic case of irrational coordinates all eigenfunctions feel the effect of the scatterer, therefore there are only ``new'' eigenfunctions.
}
in the sense that $z_1 a,z_2/a\notin\Q$. This ensures that $\delta_n>0$ for all
$n\in\scrN$. 

{\bf Clumping:} The proof of \cite{RU2} can easily be modified for rectangles with Dirichlet boundary conditions. Thus we obtain the analogue of Theorem \ref{clumping} for the operator $-\Delta_\varphi^D$.

In order to construct localized semiclassical measures we
pick
a subsequence $n\in\scrN'\subset\scrN$, of density $1-\epsilon$ for any $\epsilon>0$, such that
$\liminf_n\delta_n=\delta>0$
\footnote
{
Such a sequence may easily be constructed by noticing that the set
$\{ (z_{1} \xi_1, z_{2} \xi_{2} ) \}_{\xi \in \scrL}$ equidistributes modulo
$[0,2\pi]^2$ if $z$ is generic in the above sense. 
}. 
For mixed monomials $e_{\zeta,k}$ we may
now apply exactly the same argument as in the proof of Theorem
\ref{mixed}, to see that $\lim_{n\in\scrN'}\left\langle
  Op(e_{\zeta,k})g_{\lambda_n}^D,g_{\lambda_n}^D\right\rangle=0$. The
analogue of Proposition \ref{QEfails} is again proved in exactly the
same way as above. Hence the analogue of Theorem \ref{MomentumScars}
for diophantine rectangles with Dirichlet boundary conditions
follows. 
\begin{remark}
\label{rem:strong-coupling-extension}
In a similar fashion we can prove the result also for the strong
coupling limit (see section 4), where within a subsequence of positive
density the eigenfunctions have positive mass on a finite number of
Dirac masses in momentum space. 
\end{remark}

\end{appendix}


\begin{thebibliography}{99}

\bibitem{Berry}
M.~Berry, {\em Regular and irregular semiclassical wave functions}, J. Phys. A 10, 2083--91, 1977.

\bibitem{BerryTabor}
M.~Berry, M.~Tabor, {\em Level clustering in the regular spectrum}, Proc. Royal Soc. Lond. A 356, 375--394, 1977.


\bibitem{BerkolaikoKeatingWinn2}
G.~Berkolaiko, J.~P.~Keating and B.~Winn, {\em No quantum ergodicity for star graphs}, Comm. Math. Phys. Vol. 250, 259--285, 2004.

\bibitem{BFN}
S.~de~Bi\`evre, F.~Faure, S.~Nonnenmacher, {\em Scarred eigenstates for quantum cat maps of minimal periods}, Comm. Math. Phys. 239, 449--492, 2003.



\bibitem{BogomolnyGerlandSchmit}
E.~Bogomolny, U.~Gerland and C.~Schmit, {\em Singular Statistics}, Phys. Rev. E, Vol. 63, No. 3, 2001.

\bibitem{BogomolnySchmit}
E.~Bogomolny, C.~Schmit, {\em Structure of Wave Functions of Pseudointegrable Billiards}, Phys. Rev. Lett. 92 (2004), No. 24, 244102.

\bibitem{Bourgain-cat}
J.~Bourgain,
{\em A remark on quantum ergodicity for CAT maps}. Geometric aspects
of functional analysis, 89–98,  
Lecture Notes in Math., 1910, Springer, Berlin, 2007. 



\bibitem{CdV2}
Y.~Colin de Verdi\`ere, {\em Ergodicit\'e et fonctions propres du laplacien.} Comm. Math. Phys. Vol. 102, No. 3, 497--502, 1985.

\bibitem{EMM}
A.~Eskin, G.~Margulis and S.~Mozes, {\em Quadratic forms of signature (2,2) and eigenvalue spacings on rectangular  tori}. Ann. of Math. (2) 161, no. 2, 679--725, 2005.

\bibitem{FaddeevBerezin}
F.~A.~Berezin and L.~D.~Faddeev, {\em Remark on the Schr\"odinger equation with a singular potential}, Dokl. Akad. Nauk SSSR 137, 1011--1014, 1961 (Russian); English translation: Soviet Mathematics 2, 372--375, 1961.





\bibitem{Hassell}
A.~Hassell, {\em Ergodic billiards that are not quantum unique ergodic}. With an appendix by A. Hassell and L. Hillairet. Ann. of Math. (2), no. 171, 605 -- 618, 2010.






\bibitem{KeatingMarklofWinn2}
J.~P.~Keating, J.~Marklof and B.~Winn, {\em Localised eigenfunctions in \v{S}eba billiards.} J. Math. Phys. 51, no. 062101.


\bibitem{kelmer}
D.~Kelmer,
{\em Arithmetic quantum unique ergodicity for symplectic linear maps
  of the multidimensional torus.}
Ann. of Math. (2) 171 (2010), no. 2, 815–879.


\bibitem{KronigPenney}
R.~de L.~Kronig and W.~G.~Penney, {\em Quantum Mechanics of Electrons in Crystal Lattices.} Proceedings of the Royal Society of London. Series A, Vol. 130, No. 814, 499--513, 1931.

\bibitem{KurlbergU}
P.~Kurlberg, H.~Uebersch\"ar, {\em Quantum ergodicity for point scatterers on arithmetic tori}, Geom. Funct. Anal. 24, No. 5, 1565--90, 2014.







\bibitem{RU}
Z.~Rudnick, H.~Uebersch\"ar, {\em Statistics of wave functions for a point scatterer on the torus}. Comm. Math. Phys., Vol. 316, No. 3, 763--782, 2012.

\bibitem{RU2}
Z.~Rudnick, H.~Uebersch\"ar, {\em On the eigenvalue spacing distribution for a point scatterer on a flat torus}. Annales Henri Poincar\'e Vol. 15, No. 1, 1--27, 2014.

\bibitem{Seba}
P.~ \v{S}eba, {\em Wave chaos in singular quantum billiard}. Phys. Rev. Lett. 64, 1855--1858, 1990.

\bibitem{SebaExner}
P.~\v{S}eba, P.~Exner, {\em Point interactions in two and three dimensions as models of small scatterers}. Physics Letters A 222, 1--4, 1996.
 
\bibitem{Shigehara1}
T.~Shigehara, {\em Conditions for the appearance of wave chaos in quantum singular systems with a pointlike scatterer},
Phys. Rev. E, Vol. 50, No. 6, 1994.

\bibitem{Shigehara2}
T.~Shigehara, T. Cheon, {\em Wave chaos in quantum billiards with a
small but finite-size scatterer}, Phys. Rev. E, Vol. 54, No. 2,
1321--1331, 1996.

\bibitem{ShigeharaCheon3D}
T.~Shigehara, T. Cheon, {\em Spectral properties of
three-dimensional quantum billiards with a pointlike scatterer},
Phys. Rev. E 55, 6832--684, 1997.

\bibitem{Shigehara3}
T.~Shigehara, H. Mizoguchi, T. Mishima, T. Cheon, {\em Chaos Induced
by Quantization}, 1998.


\bibitem{Stoeckmann}
T. Tudorovskiy, U. Kuhl, H-J. St\"ockmann, {\em Singular statistics revised}, New J. Phys. 12, 2010.





\end{thebibliography}
\end{document}